\documentclass[preprint]{imsart}

\usepackage{amsmath}
\usepackage{amsfonts}
\RequirePackage{amsthm,amsmath,natbib}
\RequirePackage[colorlinks,citecolor=blue,urlcolor=blue]{hyperref}
\RequirePackage{hypernat}
\usepackage{amssymb}
\usepackage{fullpage}
\newtheorem{theorem}{Theorem}

\newcommand{\N}{{\mathbb N}}

\newcommand{\comment}[1]{} 
\begin{document}
\title{Moment densities of super-Brownian motion,  and a Harnack estimate for a class of $X$-harmonic functions}
\author{Thomas S. Salisbury \and A. Deniz Sezer}
\thanks{Both authors are supported in part by NSERC. This version \today. MSC 60J68 (primary); 60J45, 31B05 (secondary)}
\address{York University and University of Calgary}
\date{\today}
\begin{abstract}
This paper features a comparison inequality for the densities of the moment measures of super-Brownian motion.  These densities are defined recursively for each $n\geq 1$ in terms of the Poisson and Green's kernels, hence can be analyzed using the techniques of classical potential theory.  When $n=1$, the moment density is equal to the Poisson kernel, and the comparison is simply the classical inequality of Harnack.  For $n>1$ we find that the constant in the comparison inequality grows at most exponentially with $n$.  We apply this to a class of $X$-harmonic functions $H^\nu$ of super-Brownian motion, introduced by Dynkin. We show that for a.e. $H^\nu$ in this class, $H^\nu(\mu)<\infty$ for every $\mu$.

\end{abstract}
\maketitle

\section{A moment density inequality }
Let $D$ be a bounded and smooth domain in $\mathbb{R}^{d}$, $d\ge 2$ and let $\Delta=\sum_{i=1}^{d} \frac{\partial^2}{\partial x_i^2}$.
Let $g_D(x,y)$ and $k_D(x,y)$ be respectively the Green's and Poisson kernels of $D$ for the operator $\Delta$.  That is, for $x\in D$, $g_D(x,\cdot):D-\{x\}\to [0,\infty)$ is the unique $C^{2}(D-\{x\})\cap C^{1} (\bar{D}-\{x\})$ solution of
\begin{eqnarray*}
\label{green}
\Delta u&=&2\delta_x \mbox{ on }D \\
u&=&0\mbox{ on }\partial{D}.\nonumber
\end{eqnarray*}
Then for $z\in\partial D$, $k_D(x,y)\in [0,\infty)$ is the derivative of $g_D(x,z)$, considered as a function of $z$, in the direction of the inward conormal $n_y$ to $\partial D$ at the point $y\in\partial{D}$.  For fixed $y\in\partial{D}$, the function $u=k_D(\cdot,y)$ satisfies
\begin{eqnarray*}
\label{poisson}
\Delta u&=&0 \mbox{ on }D \\
u&=&0 \mbox{ on }\partial{D}-\{y\}.\nonumber
\end{eqnarray*}
For $y\in D$, $\lim_{x\rightarrow y}g_D(x,y)=\infty$. 
Similarly,  for fixed $y\in\partial{D}$, $\lim_{x\rightarrow y} k_D(x,y)=\infty$, provided the limit is taken non-tangentially. (See \cite{Doob84} for these and other
basic results in potential theory.) The notation $D'\Subset D$ will mean that $D'$ has compact closure which is contained in $D$.

Let $n\geq 1$, and let $z_1,\ldots,z_n$ be distinct points on $\partial{D}$.  We will
define a function $\rho_D(x,z_1,\ldots,z_n)$, $x\in D$ recursively, by defining a family of functions $\rho_D(x,z_A)$, where $A=\{i_1,\ldots,i_k\}\subset \{1,\ldots,n\}$ and  $z_A=(z_{i_1},\ldots,z_{i_k})$. Set
 \begin{eqnarray}
\label{fom}
\rho_D(x,z_A)&=&k_D(x,z_i)\mbox{ if $A=\{i\}$,}
\end{eqnarray}
and for $|A|>1$ set
\begin{eqnarray}
\label{recursion}
\rho_D(x,z_A)&=&\sum_{B\subset A, B\neq \emptyset, A} \int_D g_D(x,y)\rho_D(y,z_B)\rho_D(y,z_{A-B})dy.
\end{eqnarray}
We let $\rho_D(x,z_1,\ldots,z_n)=\rho(x,z_A)$ for $A=\{1,\ldots,n\}$.  It is known that $\rho_D(x,z_1,\ldots,z_n)$ is a finite valued function of $x\in D$.  See, for example \cite{Dyn04b} or \cite{SV}. In section \ref{SBMproofs} we will interpret $\rho$ as a moment density. Unless needed for clarity we will drop the subscript $D$ in $\rho_D$, $k_D$ and $g_D$. 

The main estimate of this section is the following theorem:

\begin{theorem} \label{theo:comparison} Let $D$ be a smooth and bounded domain in $\mathbb{R}^d$, $d\geq 2$.  For any compact $C\subset D$, there exists a $\lambda>0$ depending only on $C$ and $D$ such that
\[\rho(x,z_1,\ldots,z_n)\leq\lambda^n \rho(x_0,z_1,\ldots,z_n)\]
for all $x$ and $x_0\in C$ and for all distinct $z_1,\ldots,z_n \in \partial D^n$.

\end{theorem}

We will describe the implications of this result in section \ref{application}.

One should compare the above estimate to Harnack's inequality: Indeed, when $n=1$, $\rho(x,z)=k(x,z)$. The theorem in this case therefore follows from Harnack's inequality, which asserts that there exists $\theta_C$ depending only on $C$ and $D$ s.t $k(x,z)\leq \theta_C k(x_0,z)$ for all $x,x_0\in C$ and $z\in \partial{D}$.   

The proof for general $n$ will be given in Section~\ref{proof}.  To see where the main challenge in the proof arises, consider the case $n=2$. Then, $\rho(x,z_1,z_2)=2 \int_D g(x,y) k(y,z_1)k(y,z_2)\,dy$. Due to Harnack's inequality, there exists a $\theta_{D'}$ such that  $g(x,y)\leq \theta_{D'} g(x_0,y) $ for any $y\in D-D'$, where $C\subset D'$, and $D'\Subset D$. If this inequality were to hold on the entire set $D$ it would be easy to compare $\rho(x,z_1,z_2)$ with $\rho(x_0,z_1,z_2)$.  However this is not the case, as $g(x,y)$ blows up when $y$ is near $x$.  Since the integral defining $\rho$ is taken over the entire domain $D$, in order to compare $\rho(x, \cdot)$ with $\rho(x_0,\cdot)$ we need finer estimates for $g(x,y)$ when $y$ is near $x$, and for $k(y,z_1)k(y, z_2)$ when $y$ is near  $z_1$ or $z_2$. 

The proof follows an inductive argument. We will obtain compact subsets $C_0\supset C_1\supset\ldots\supset C$ of $D$. Then for each $n$, we will prove that  
\[\rho(x,z_1,\ldots,z_n)\leq \lambda^n \rho(x_0,z_1,\ldots,z_n)\] for all $x,x_0\in C_{n}$, assuming 
\[\rho(x,z_1,\ldots,z_{n-1})\leq\lambda^{n-1} \rho(x_0,z_1,\ldots,z_{n-1})\]
for all $x,x_0\in C_{n-1}$. We do this by decomposing $\int_D g(x,y)\rho(y,z_B)\rho(y,z_{B^c})dy$ as 
 \[\int_{B(x,\delta_n)}g(x,y)\rho(y,z_B)\rho(y,z_{B^c})dy +\int_{D-B(x,\delta_n)}g(x,y)\rho(y,z_B)\rho(y,z_{B^c})dy\]
where $\delta_n=\frac{\delta}{Nn^2}$. Here $\delta$ is the diameter of the domain and $N$ is a certain constant strictly greater than $1$.  
For the first term, we use the induction hypothesis and the fact that $g(x,y)$ behaves like $|x-y|^{2-d}$ when $y$ is near $x$ to show that
\[\int_{B(x,\delta_n)}g(x,y)\rho(y,z_B)\rho(y,z_{B^c})dy \leq \frac{\lambda^{n}}{2}\int_{B(x_0,\delta'_n)}g(x_0,y)\rho(y,z_B)\rho(y,z_{B^c})dy\]
for some $\lambda>0$.
For the second term we use the 3-G inequality. In the case $d\ge 3$ this gives a $\theta$ depending only on $D$ such that  \[\frac{g(x,y)}{g(x_0,y)}\leq \theta \frac{|x-y|^{2-d}+|x-x_0|^{2-d}}{g(x,x_0)}.\] The critical factor on the right side is $|x-y |^{2-d}$, which results in a bound for $\frac{g(x,y)}{g(x_0,y)}$ on $D-B(x,\delta_n)$ of the order of $n^{2d}$ which can in turn be bounded by $\lambda^n$ if $\lambda$ is large enough. This implies that 
\[\int_{D-B(x,\delta_n)}g(x,y)\rho(y,z_B)\rho(y,z_{B^c})dy 
\leq \frac{\lambda^n}{2}   \int_{D-B(x,\delta_n)}g(x,y)\rho(y,z_B)\rho(y,z_{B^c})dy.\]
Putting these two pieces together will prove the theorem. 

The paper is organized as follows. In Section \ref{application} we will describe an application of Theorem \ref{theo:comparison} to  a finiteness question for an interesting class of extended $X$-harmonic functions related to super-Brownian motion (originally introduced by Dynkin). Section \ref{SBMproofs} contains the proof of this result, and the proof of Theorem \ref{theo:comparison} is given in Section \ref{proof}.

\section{$X$-harmonic functions}\label{application}

The functions $\rho_D(x,z_1,\ldots z_n)$ will arise as moment densities of super-Brownian motion.  More precisely, we consider a super-Brownian motion on some domain $E$ (larger than all domains $D$ of interest), as represented by Dynkin in terms of exit measures.  See e.g. \cite{Dyn02} for the definition and construction. This is a branching-exit-Markov system $(X_{D}, P_{\mu})$ indexed by sub-domains $D$ of $E$, and by finite measures $\mu$ supported in $E$. Under the probability law $P_\mu$, each $X_D$ is a random finite measure, referred to as the \emph{exit measure} from $D$ of super-Brownian motion started with measure $\mu$. Heuristically, in a branching particle approximation to super-Brownian motion, it would be approximated by a multiple of the empirical measure of the particles frozen when they exit $D$. More formally, it has the following basic properties:

\begin{enumerate}
\item Exit property: $P_{\mu}(X_{D}(D)=0)=1$ 
for every $\mu$, and if $\mu(D)=0$ then $P_{\mu}(X_{D}=\mu)=1$. 
\item Markov property:  If
$Y\ge 0$ is measurable with respect to the $\sigma$-algebra $\mathcal{F}_{\subset D}$ generated by $X_{D'},D'\subset D$ and $Z\geq 0$ is measurable with respect to the $\sigma$-algebra $\mathcal{F}_{\supset D}$ generated by $X_{D''}$, $D''\supset D$ then 
\[P_\mu(YZ)=P_\mu(YP_{X_{D}}Z).\]
\item Branching property:   For any non-negative Borel $f$, $P_\mu(e^{-\langle X_{D},f\rangle})=e^{-\langle \mu, V_Df\rangle}$ where 
\[V_Df(y)=-\log P_{y}(e^{-\langle X_D,f\rangle})\]
and $P_y=P_{\delta_y}$.
\item Integral equation for the log-Laplace functional:  $V_{D}f$ solves the integral equation
\[ u+G_D(2u^2)=K_D f\]
where $G_D$ and $K_D$ are respectively Green and Poisson operators for Brownian motion in $D$. 
\end{enumerate}
The operators referred to above are defined as follows: If $\xi_t$ is a Brownian motion starting from $x$, under a probability measure $\Pi_{x}$, 
then $K_{D}f(x)=\Pi_xf(\xi_{\tau_D})$, where $\tau_D$ is the exit time from $D$.  Likewise, $G_{D}f(x)=\Pi_x(\int_{0}^{\tau_D}f(\xi_t)dt)$.  
In terms of the kernels $g_D$ and $k_D$ defined earlier, $G_Df(x)=\int_Dg_D(x,y)\,dy$ and $K_Df(x)=\int f(y)k_D(x,y)\sigma(dy)$, where $\sigma$ is the surface measure on $\partial D$. 

Under certain regularity conditions on $D$ and $f$ (see e.g. \cite{Dyn02}), the integral equation in (d) is equivalent to the boundary value problem    
\begin{eqnarray*}
\frac{1}{2}\Delta u&=&2u^2\mbox{ on $D$}\\ 
u&=&f\mbox{ on $\partial D$}. 
\end{eqnarray*}
   
A recent research program, initiated by Dynkin, aims to build a Martin boundary theory for the above semi-linear pde employing ideas from the probabilistic construction of Martin boundary for the Laplace equation $\frac{1}{2}\Delta u=0$.  In this new context, the analogue of harmonic functions are $X$-harmonic functions, defined as follows.  Let $\mathcal{M}_{D}^c$ denote the space of finite measures supported on a compact subset of $D$. Consider a function $H:\mathcal{M}_{D}^c\to [0,\infty]$.  $H$ is called \emph{extended $X$-harmonic} in $D$ if
\[P_{\mu}(H(X_{D'}))=H(\mu)\]
for all $D'\Subset D$ and $\mu$ s.t. $\mu$ is supported in $D'$.
If, in addition $H(\mu)$ is finite for all $\mu\in \mathcal{M}_{D}^c$, then $H$ is called \emph{$X$-harmonic} (this definition is due to Dynkin -- see \cite{Dyn04b}). 

Define 
\[P_{\mu,X_D}(A):=P_{\mu}(X_D\in A, X_D\neq 0).\]
Theorem 5.3.2 of \cite{Dyn04a} shows that $P_{\mu_1,X_D}$ and $P_{\mu_2,X_D}$ are mutually absolutely continuous. We fix $x_0\in D$, and take $R=P_{\delta_{x_0},X_D}$ as a reference measure. Following 
\cite{Dyn06}, we define
\begin{equation}\label{RNdensity}
H^\nu(\mu)=H^\nu_D(\mu)=\frac{dP_{\mu,X_D}}{dR}(\nu).
\end{equation}
Unless there is ambiguity about the domain $D$ in question, we will suppress the subscript $D$ in $H^\nu_D$. 

Of course, for each $\mu$ this is only uniquely defined for $R$-a.e. $\nu$, so to proceed further we will need to specify more regular versions. The following is contained in Theorem 2 of \cite{SS} (whose proof is a modification of that of Theorem 1.1 of \cite{Dyn06}), and establishes the existence of a jointly measurable version of $H^{\nu}(\mu)$ that is extended $X$-harmonic for $R$-almost all $\nu$. 
We assume that $D$ is a bounded domain all of whose boundary points are regular.  The smoothness assumption of Theorem \ref{theo:comparison} is not required here.  
\begin{theorem}(\cite{Dyn06}, \cite{SS}\label{SSresult})
Let $D$ be a bounded and regular domain in $\mathbb{R}^d$, $d\ge 2$.  There exists a measurable function $(\mu,\nu)\mapsto H^\nu(\mu)\in[0,\infty]$, defined for finite $\mu\in \mathcal{M}_{D}^c$ and finite $\nu$ supported on $\partial D$, and an $R$-null set $V_0$ such that:
\begin{enumerate}
\item For each $\nu\notin V_0$, $\mu\mapsto H^\nu(\mu)$ is extended $X$-harmonic.
\item For each $\mu$, $\nu\mapsto H^\nu(\mu)$ is a version of the density in \eqref{RNdensity}.
\end{enumerate}
\end{theorem}
For each fixed $\mu$, (b) implies that $H^{\nu}(\mu)<\infty$ for $R$-a.e. $\nu$. But the $R$-null set where $H^{\nu}(\mu)=\infty$ could in principle depend on $\mu$, so one does not automatically obtain any $\nu$ for which $H^{\nu}(\mu)<\infty\,\forall\mu$. In other words, it is not clear that any of the $H^\nu$ is $X$-harmonic. However, we can establish this using Theorem \ref{theo:comparison}:
\begin{theorem}\label{extendedSSresult}
Let $D$ be a bounded and regular domain in $\mathbb{R}^d$, $d\ge 2$.  There exists an $H^\nu(\mu)$ as in Theorem \ref{SSresult}, and an $R$-null set $V_1$ such that for $\nu\notin V_1$, $H^\nu$ is $X$-harmonic, and $0<H^\nu(\mu)<\infty$ for every $\mu$.
\end{theorem}

The above is the principle result of this section, and the main application of Theorem \ref{theo:comparison}. We give the proof in section \ref{extendedSSresult}, and will spend the remainder of the current section setting the context for this result, and giving relations to other work.

Note first that the statement of Theorem 1.1 of \cite{Dyn06} is stronger (ie. $X$-harmonic) than than what is verified in the proof there (ie. what we call extended $X$-harmonic). Obtaining the required finiteness estimate was the motivation for the current paper. 

If $H$ is a non-zero $X$-harmonic function one can proceed to define an {\it $H$-transform} of the law $P_{\mu}$. Set $P_{\mu}^H(Y)=\frac{1}{H(\mu)}P_{\mu}(YH(X_{D'}))$ for any non-negative $Y$ measurable with respect to $\mathcal{F}_{D'}=\sigma(X_{\tilde{D}},  \tilde{D}\subset D')$ where $D'\Subset D$ contains the support of $\mu$.   Because $H$ is $X$-harmonic, the Markov property implies that these probability laws will be consistent as we vary $D'$, so can be uniquely extended to $\mathcal{F}_{D-}=\sigma(X_{D'},  D'\Subset D)$, as in Section 1.3 of \cite{Dyn06}.
See also Theorem 2(d) of \cite{SS}. Note that the mutual absolute continuity of $P_{\mu_1,X_D}$ and $P_{\mu_2,X_D}$, for $\mu_1$ and $\mu_2$ compactly supported in $D$, implies that if $H(\mu)$ is non-zero from one $\mu$ then it is non-zero for all $\mu$. Thus the only obstacle to defining $P_{\mu}^H$ for $H$ extended $X$-harmonic is the possibility that $H(\mu)=\infty$.

Fix $a\in D$, $a\neq x_0$ and let $\mathcal{H}^a$ be the convex set of all $X$-harmonic functions s.t. $P_{\delta_a}(H(X_{D'}))=1$ for some $D'$ with $a\in D'\Subset D$. The (minimal) Martin boundary is then defined as the set of extreme elements of $\mathcal{H}^{a}$. In \cite{Dyn04b} it is shown that this set is independent of the reference point $a$. One motivation for introducing the Radon-Nikodym densities $H^\nu(\mu)$ of \cite{Dyn06} was to try to recover the extreme $X$-harmonic functions by a deterministic limiting procedure from these densities, using the classical exit theory for Markov chains.  What is actually known is a weak version of such a result. Specifically, Dynkin succeeded in showing that if $H$ is an extreme $X$-harmonic function in $D$ and $D_n\Subset D$ is a sequence of domains exhausting $D$, then $H(\mu)=\lim_{n\rightarrow\infty}H^{X_{D_n}}_{D_n}(\mu)$, $P_{\mu}^H$ almost surely. A stronger version is conjectured, namely that there exists a deterministic sequence $\nu_n$ (not depending on $\mu$) and canonical versions of the $H^{\nu_n}_{D_n}$ (eg. having some type of regularity in $\nu_n$), such that $H(\mu)=\lim_{n\rightarrow\infty}H^{\nu_n}_{D_n}(\mu)$ for every $\mu$. Dynkin has pointed out that significant progress could be made if one knew that the densities $H^{X_{D_n}}_{D_n}$ were uniformly integrable. A subsequent paper \cite{DynkinIJM06}  focused on the densities $H^{\nu}_{D}(\mu)$ and the search for more analytically tractable representations. 

The densities $H^{\nu}_{D}$ played a key role in \cite{SS} as well, in particular, in the analysis of super-Brownian motion $(X_{D'})_{D'\Subset D}$ conditioned on its exit measure $X_D$.  In addition to Theorem \ref{SSresult}, that paper showed that $P_\mu^{H^\nu_D}$ is the conditional law of SBM given $X_D=\nu$, for $R$-a.e. $\nu$ satisfying $H^{\nu}_D(\mu)<\infty$.  Assuming that $D$ is a regular domain, they express $H^{\nu}_D$ as an infinite sum of integrated polynomial-like functions of $\mu$ involving a family of fragmentation kernels $K_n(\nu,d\nu_1,\ldots,\nu_n)$ and a family of potentials $\gamma_{\nu}(\cdot)$, (see formula (\ref{eq:linear}) below). This representation is then used to derive an infinite fragmentation system description for $P_\mu^{H^\nu_D}$.

In the current paper we will use the representation of \cite{SS} to tie the densities $H^{\nu}_D$ to the functions $\rho(x,z_,\ldots,z_n)$ and prove Theorem \ref{extendedSSresult}.  This suggests that the comparison inequality of Theorem \ref{theo:comparison} may prove useful in investigating other regularity  properties of $H^{\nu}_D$.   In a subsequent paper, one of us (Sezer) will study the extremality properties of this family for smooth domains.  As a future direction of research, we hope to use the comparison inequality to investigate uniform integrability properties of the sequences $H^{\nu_n}_{D_n}$ of \cite{Dyn06}.  

We conjecture that a much stronger version of Theorem \ref{extendedSSresult} is true. Namely that if $H$ is any extended $X$-harmonic function, and $H(\mu)<\infty$ for some $\mu$, then $H(\mu)<\infty$ for all $\mu$.

\section{Proof of Theorem \ref{extendedSSresult}}
\label{SBMproofs}
   
It is well known that $X_{D}$ has an infinitely divisible distribution for each $D$.  This property is tied to a construction of the probabilities $P_\mu$ in terms of another (now $\sigma$-finite) measure $\mathbb{N}_{x}$, called the \emph{super-Brownian excursion law}.  Once again, there are finite exit measures $X_D$ under $\mathbb{N}_{x}$, for every $D$. 
Form a Poisson random measure 
\[\Pi(d\chi)=\sum_{i}\delta_{\chi^{i}}\]
with intensity $\mathcal{R}_\mu(A)=\int \N_{x}(X_D\in A,X_D\neq 0)\mu(dx)$. Then under $P_\mu$, $X_D$ has the law of $\sum\chi^i=\int\chi\Pi(d\chi)$.

The $n$-dimensional moment measures of SBM are the following measures defined on $(\partial D)^n$:
Let $f(z_1,\ldots,z_n)=f_1(z_1)\ldots f_n(z_n)$ where $f_i$ are positive Borel functions on $\partial D$, and define
\begin{align}
\label{eq:nmoment}
N_{x,n,D}(f)&=\mathbb{N}_{x}(\langle
X_D,f_1\rangle\cdots \langle
X_D,f_n\rangle )\\
\label{eq:pmoment}
p_{\mu,n,D}(f)&=P_{\mu}(\langle X_D,f_1 \rangle\cdots \langle
X_D,f_n\rangle ).
\end{align}

\begin{theorem}[Theorem 5.3.1 of \cite{Dyn04a}] \label{densities}Let $x\in D$, and let $\mu$ be compactly supported in $D$. If $D$ is smooth then 
both $N_{x,n,D}$ and 
$p_{\mu,n,D}$ have densities with respect to $\sigma(dz_1)\times \sigma(dz_2)\times\cdots\times \sigma(dz_n)$ where $\sigma(dz)$ denotes the surface measure on $\partial D$.  Moreover,  
\begin{align*}
N_{x,n,D}(dz_1,\ldots ,dz_n)&=\rho(x,z_1,\ldots,z_n)\sigma(dz_1) \sigma(dz_2)\cdots \sigma(dz_n)\\
p_{\mu,n,D}(dz_1,\ldots,dz_n)&=\sum_{\pi(n)}\langle\mu,\rho(\cdot,z_{C_1})\rangle\cdots\langle\mu,\rho(\cdot,z_{C_r})\rangle
\sigma(dz_1) \sigma(dz_2)\cdots \sigma(dz_n)
\end{align*}
where $\pi(n)$ denotes the set of partitions $\{C_1,\dots,C_r\}$ of $\{1,\ldots,n\}$.
\end{theorem}
The proof of Theorem \ref{extendedSSresult} involves only $p_{\mu,n,D}$, not $N_{x,n,D}$. But the latter gives the natural context for directly interpreting $\rho$. 

Let $V^*$ be the space of nonzero and finite measures on $\partial{D}$. Let $u(x)=\N_x(X_D\neq 0)$. 
For a regular domain,  \cite{SS} construct a family of functions $\{\gamma_{\nu}:D\to (0,\infty), \nu\in V^*\}$ such that the mapping $(\nu,y)\mapsto \gamma_{\nu}(y)$ is measurable, each $\gamma_\nu$ is superharmonic,  and for all $y\in D$
\begin{equation*}
N_{y,X_{D}} (d\nu):=\N_{y}(X_D\in d\nu, X_D\neq 0)=\gamma_{\nu}(y)R(d\nu) \label{eq:gamma}.
\end{equation*}
In addition,  \cite{SS} construct 
a measurable strictly positive kernel 
${K}_{n}(\nu;d\nu_1,d\nu_2,\dots,d\nu_{n})$
from $V^*$ to $(V^*)^{n}$, concentrated on $\{(\nu_1,\dots,\nu_n): \nu_1+\dots+\nu_n=\nu\}$, and an $R$-null set $V_0$, such that
\begin{eqnarray}\label{eq:linear}
H^{\nu}(\mu)=\left\{\begin{array}{l}e^{-\langle\mu,u\rangle+u(x_0)}, \mbox{ if $\nu=0$}\\
\sum_{n=1}^{\infty}\int \frac{1}{n!}e^{-\langle\mu,u\rangle}K_{n}(\nu;d\nu_1,\dots,d\nu_{n})
\langle\mu,\gamma_{\nu_1}\rangle\langle\mu,\gamma_{\nu_2}\rangle\cdots\langle\mu,\gamma_{\nu_n}\rangle,\quad
\mbox{if $\nu\neq 0$}\end{array}\right.
\end{eqnarray}
is extended $X$-harmonic for each $\nu\notin V_0$, is a version of the density in \eqref{RNdensity} for each $\mu$, and also satisfies
\begin{equation*}
\gamma_{\nu}(y)=\N_{y}(H^{\nu}(X_{D'}))\label{eq:version}
\end{equation*}
for every $y$, every $D'\Subset D$ such that $y\in D'$, and every $\nu\notin V_0$. 

\begin{proof}[Proof of Theorem \ref{extendedSSresult}] Fix $a>0$. It suffices to show the statement of the theorem for $\mu$ satisfying $\langle\mu,1\rangle\le a$. We may assume $\nu\neq 0$. 

Because $D$ is regular there exists a sequence of smooth domains $D_1\Subset D_2\Subset \ldots \Subset D$ exhausting 
$D$.
Let $x_0$ be the reference point fixed earlier, and let $\mu_0$ be the point mass at $x_0$.  Let $\mu$ be another finite measure compactly supported in $D$.  Then there exists a smooth domain $\tilde{D}\in \{D_1,D_2,\ldots\}$ containing the point $x_0$ and the support of $\mu$. Choose $H^\nu(\mu)$, $V_0$, $\gamma_\nu$ and $K_n$ as above. 

Let $\rho_{\tilde{D}}(x,z_1,\ldots,z_n)$, and $\rho_{\tilde{D}}(x,z_C), C\subset\{1,\ldots,n\},n\geq 1$ be the family of functions defined by equations (\ref{fom}) and (\ref{recursion}).
Let $p_{\mu,n,\tilde{D}}$ be the measure on $\partial{\tilde D}$ defined as in (\ref{eq:pmoment}).  Note that by Theorem \ref{densities},
\begin{equation}\label{radonnikodym}
\frac{dp_{\mu,n,\tilde{D}}}{dp_{\mu_0,n,\tilde{D}}}(z_1,\ldots,z_n)=\frac{\sum_{\pi(n)}\langle\mu,\rho(\cdot,z_{C_1})\rangle\cdots\langle\mu,\rho(\cdot,z_{C_r})\rangle}{\sum_{\pi(n)}\rho(x_0,z_{C_1})\cdots\rho(x_0,z_{C_r})}.
\end{equation}
By Theorem \ref{theo:comparison} there exists $\lambda>0$ depending only on $\tilde{D}$ and $C\cup\{x_0\}$ such that 
\begin{equation}
\rho(x,z_{C_i})\leq \lambda^{|C_i|}\rho(x_0,z_{C_i})\mbox{ for every $x\in C$.} \label{inequality1}
\end{equation}
Equation (\ref{radonnikodym}) and the inequality (\ref{inequality1}) imply that
\begin{equation}\label{radonnikodym2}
\frac{dp_{\mu,n,\tilde{D}}}{dp_{\mu_0,n,\tilde{D}}}\leq (a\lambda)^n.
\end{equation}
Let $\tilde{H}_{x}^{\nu}(\mu)=e^{\langle\mu,u\rangle}H^{\nu}_x(\mu)$.  Using the formula \eqref{eq:linear}, the Fubini theorem, and \eqref{radonnikodym2}, we obtain
\begin{eqnarray*}
P_{\mu}(\tilde{H}_{x}^{\nu}(X_{\tilde{D}}))&=&\sum_{n=1}^{\infty}\int \frac{1}{n!}K_{n}(\nu;d\nu_1,\dots,d\nu_{n})
P_{\mu}\left(\langle X_{\tilde{D}},\gamma_{\nu_1}\rangle\cdots\langle{X_{\tilde{D}},\gamma_{\nu_{n}}}\rangle\right)\nonumber\\
&=&\sum_{n=1}^{\infty}\int \frac{1}{n!}K_{n}(\nu;d\nu_1,\dots,d\nu_{n})
p_{\mu,n,\tilde{D}}\left(\gamma_{\nu_1}\cdots\gamma_{\nu_{n}}\right)\nonumber\\
&=&\sum_{n=1}^{\infty}\int \frac{1}{n!}K_{n}(\nu;d\nu_1,\dots,d\nu_{n})
p_{\mu_0,n,\tilde{D}}\left(\frac{dp_{\mu,n,\tilde{D}}}{dp_{\mu_0,n,\tilde{D}}}\gamma_{\nu_1}\cdots\gamma_{\nu_{n}}\right)\nonumber\\
&\leq&\sum_{n=1}^{\infty}\int \frac{1}{n!}K_{n}(\nu;d\nu_1,\dots,d\nu_{n})
p_{\mu_0,n,\tilde{D}}\left((a\lambda)^n\gamma_{\nu_1}\cdots\gamma_{\nu_{n}}\right)\nonumber\\
&=&\sum_{n=1}^{\infty}\int \frac{1}{n!}K_{n}(\nu;d\nu_1,\dots,d\nu_{n})P_{\mu_0}\left(\langle a\lambda X_{\tilde{D}},\gamma_{\nu_1}\rangle\cdots\langle a\lambda X_{\tilde{D}},\gamma_{\nu_{n}}\rangle\right)\nonumber\\
&=&P_{\mu_0} (\tilde{H}^{\nu} (a\lambda X_{\tilde{D}})).\label{Pmu}
\end{eqnarray*}

We will now show that there exists a $V_1\subset V^*$ such that $V_0\subset V_1$, $R(V_1)=0$, and for all $\nu\notin V_1$, $P_{\mu_0} (\tilde{H}^{\nu} (k X_{D_i}))<\infty$ for all integers $k\geq 1$ and $i\geq 1$. This will establish the theorem, since 
for all $\nu\notin V_1$, 
\[H^{\nu}(\mu)=P_{\mu}(H^{\nu}(X_{\tilde{D}}))\leq  P_{\mu}(\tilde{H}^{\nu}(X_{\tilde{D}}))\leq P_{\mu_0}(\tilde{H}^{\nu}(a\lambda X_{\tilde{D}}))\le P_{\mu_0}(\tilde{H}^{\nu}(k X_{\tilde{D}}))<\infty\]
whenever $k\ge a\lambda$. Note that $H^\nu(\mu)>0$ by definition, since the $\gamma_{v_i}>0$. 

Fix $i$ and $k$.  Let $M>0$. For fixed $\mu$, $H^{\nu}(\mu)$, $\nu\neq 0$ is a version of $\frac{dP_{\mu,X_D}}{dR}(\nu)$, so
\begin{eqnarray*}
\int_{V^*}e^{-\langle \nu, M\rangle}P_{\mu_0}(\tilde{H}^{\nu}(k X_{D_i}))R(d\nu)&=& P_{\mu_0}(\int_{V^*}e^{-\langle \nu, M\rangle}\tilde{H}^{\nu}(k X_{D_i})R(d\nu))\\
&=& P_{\mu_0}(e^{\langle k X_{D_i},u\rangle} \int_{V^*}e^{-\langle \nu, M\rangle}dP_{k X_{D_i},X_D}(d\nu)\\
&=& P_{\mu_0}(e^{\langle k X_{D_i},u\rangle} P_{k X_{D_i}} (e^{-\langle X_D, M\rangle}1_{X_D\neq 0}))\\
& \leq & P_{\mu_0}(e^{\langle k X_{D_i},u\rangle} P_{k X_{D_i}} (e^{-\langle X_D, M\rangle}))\\
&=& P_{\mu_0}(e^{\langle k X_{D_i},u\rangle} e^{-\langle k X_{D_i}, u_M\rangle}),
\end{eqnarray*}
where $u_M$ is the function on $D$ defined by $u_M(x)=\log P_x( e^{-\langle X_{D}, M\rangle})$.

The above derivation shows that $\int_{V^*}e^{-\langle \nu, M\rangle}P_{\mu_0}(\tilde{H}^{\nu}(k X_{D_i}))R(d\nu)\leq P_{\mu_0}(e^{\langle X_{D_i},\lambda (u-u_M)\rangle})$. This is in fact finite if $M$ is sufficiently large.  To see this we use two facts: First, that $u_M$ converges to $u$ uniformly on the closure of $\tilde{D}$, as they are both continuous functions and $u_M\uparrow u$ pointwise in $D$ as $M\rightarrow \infty$.  Second, it is proven in \cite{SV} that there exists an $\epsilon>0$ such that 
$P_{x_0}(e^{\langle X_{D_i},\epsilon\rangle})<\infty$.  Hence choosing $M$ large enough that $\sup_{\bar{\tilde{D}}}u-u_M<\frac{\epsilon}{k}$ will ensure that  $\int_{V^*}e^{-\langle \nu, M\rangle}P_{\mu_0}(\tilde{H}^{\nu}(k X_{D_i}))R(d\nu)$ is finite.  

Because $e^{-\langle \nu, M\rangle}$ is always nonzero, it follows that $P_{\mu_0}(\tilde{H}^{\nu}(k X_{D_i}))<\infty$, for $R$-a.e. $\nu$.  
Now let $V_1$ be the subset of $V^*$ defined as 
\[V_1=V_0\cup \bigcup_{i=1}^{\infty}\bigcup_{k=1}^{\infty}\{\nu\in V^*:P_{\mu_0}(\tilde{H}^{\nu}(k X_{D_i}))=\infty\}.\]
Then $R(V_1)=0$, as required.  \end{proof}

\section{Proof of Theorem \ref{theo:comparison}}\label{proof}

We give the proof first for $d\ge 3$, and then describe how to modify it for $d=2$. 
Let $\delta=\mbox{dist}(C,\partial{D})$.  We can find a sequence of smoothly bounded domains $D_n$ with compact closures $C_n$ such that $D\supset C_0\supset C_1\supset C_2\supset\ldots \supset C$ and
\begin{equation}
\mbox{dist}(C_{n+1},C_n^c)> \frac{\delta}{(n+1)^2}, \mbox{ for all $n\geq 1$.}\label{assumption2}
\end{equation}
From Harnack's inequality there exists a constant $\phi>0$ depending only on $C_0$ and $D$ such that
\begin{equation}
k(x,z)\leq\phi k(x_0,z)\mbox{for all $x,x_0\in C_0$, and $z\in \partial{D}$}.\label{assumption3}
\end{equation}
The 3-G inequality of \cite{CFZ} gives a constant $\theta>0$ depending only on $D$ such that
\begin{equation}
\frac{g(x,y)}{g(x_0,y)}\leq \theta \frac{|x-y|^{2-d}+|x-x_0|^{2-d}}{g(x,x_0)}\mbox{for all $x,x_0,y\in D$.}\label{assumption4}
\end{equation}
Set $\beta=\text{dist}(C_0,\partial D)\land 1$. Recall that there is a constant $c_d$ such that $g(x,y) =c_d|x-y|^{2-d}-h_x(y)$ where $h_x$ is harmonic on $D$ with boundary value $c_d|x-\cdot|^{2-d}$.
Therefore 
\begin{align}
g(x,y)&\le c_d|x-y|^{2-d}\text{ for $x,y\in D$, and}\label{g_upperbound}\\
g(x,y)&\ge c_d|x-y|^{2-d}-c_d\beta^{2-d}\text{ if $x,y\in D$, $B(x,\beta)\subset D$,}\label{g_lowerbound}
\end{align}
since $0\le h_x\le c_d\beta^{2-d}$. This will imply the existence of finite constants $K$ and $B$, depending only on $D$ and $C_0$, such that
\begin{align}
\frac{1}{g(x,x_0)}&\leq B\text{ and}\label{assumption5b}\\
\frac{c_d|x-x_0|^{2-d}}{g(x,x_0)}&\leq K\text{, for all $x,x_0\in C_0$, $x\neq x_0$.}\label{assumption5a}
\end{align}
To see this, note that $g(x,x_0)$ is jointly continuous off the diagonal, and everywhere $>0$. Thus there is a strictly positive lower bound for $g$ on $C_0\times C_0$ less any neighbourhood of the diagonal. \eqref{g_lowerbound} gives such a bound on a neighbourhood of the diagonal in $C_0\times C_0$. Thus \eqref{assumption5b} holds. 

To see \eqref{assumption5a}, if $x,x_0\in C_0$ and $|x-x_0|\ge \frac{\beta}{2}$ then 
$$
g(x,x_0)\ge \frac{1}{B}\ge \frac{1}{Bc_d}\Big(\frac{\beta}{2}\Big)^{d-2}\cdot c_d|x-x_0|^{2-d}.
$$
While if $x,x_0\in C_0$ and $|x-x_0|\le \frac{\beta}{2}$ then 
$$
g(x,x_0)[1+Bc_d\beta^{2-d}]\ge g(x,x_0)+c_d\beta^{2-d}\ge c_d|x-x_0|^{2-d}
$$
by \eqref{g_lowerbound}. Therefore \eqref{assumption5a} holds with $K=\max(Bc_d(\beta/2)^{2-d},1+Bc_d\beta^{2-d})$. 

Take $N>1$ to be a number whose value will be determined later. Let $A=\{1,2,\ldots,n\}$, where $n=|A|\geq 2$. Let $x$ and $x_0$ be two distinct points in $C_n$.  Note that
\begin{eqnarray}\label{decomposition}
\int_D g(x,y)\rho(y,z_B)\rho(y,z_{A-B})dy&=&\int_{B(x,\frac{\delta}{Nn^2})} g(x,y)\rho(y,z_B)\rho(y,z_{A-B})dy \nonumber\\
&&\mbox{ }+\int_{D- B(x,\frac{\delta}{Nn^2})}g(x,y)\rho(y,z_B)\rho(y,z_{A-B})dy.
\end{eqnarray} 
We will obtain an estimate for each term in the above decomposition.

We start with the second term.  By the 3-G inequality (\ref{assumption4}), we have the bound
\[g(x,y)\leq g(x_0,y) \frac{\theta}{g(x_0,x)} \left(|x-x_0|^{2-d}+\left(\frac{Nn^2}{\delta}\right)^{d-2}\right)\mbox{for $y\in D- B(x,\frac{\delta}{Nn^2})$.}\]
Combining this with the estimates (\ref{assumption5b}) and (\ref{assumption5a}), we get
\[g(x,y)\leq g(x_0,y) \theta \left(\frac{K}{c_d}+\frac{BN^{d-2}n^{2(d-2)}}{\delta^{d-2}}\right)\mbox{for $y\in D- B(x,\frac{\delta}{Nn^2})$}.\]
Writing $\tilde{K}=\theta \frac{K}{c_d}$ and $\tilde{B}=\theta \frac{BN^{d-2}}{\delta^{d-2}}$, we have
\begin{eqnarray}
\int_{D- B(x,\frac{\delta}{Nn^2})}g(x,y)\rho(y,z_B)\rho(y,z_{A-B})dy&\leq& \left(\tilde{K}+\tilde{B} n^{2(d-2)}\right)\int_{D- B(x,\frac{\delta}{Nn^2})}g(x_0,y)\rho(y,z_B)\rho(y,z_{A-B})dy \nonumber\\
&\leq&\left(\tilde{K}+\tilde{B} n^{2(d-2)}\right)\int_Dg(x_0,y)\rho(y,z_B)\rho(y,z_{A-B})dy. \label{secondterm}
\end{eqnarray}

To treat the first term in the decomposition (\ref{decomposition}), we introduce a new variable $y'=x_0+M(y-x)$, $y\in B(x,\frac{\delta}{Nn^2})$, for some  $M\ge 1$ whose value will be determined later. Since $y=x+\frac{1}{M}(y'-x_0)$, we have that $dy=M^{-d} dy'$, and $y\in B(x, \frac{\delta}{Nn^2})$ if and only if $y'\in B(x_0, \frac{M\delta}{Nn^2})$.  So by \eqref{g_upperbound}
\begin{align}
\int_{B(x,\frac{\delta}{Nn^2})} g(x,y)\rho(y,z_B)\rho(y,z_{A-B})dy
&\leq c_d\int_{B(x,\frac{\delta}{Nn^2})}|x-y|^{2-d} \rho(y,z_B)\rho(y,z_{A-B})dy\nonumber\\
&=c_d\int_{B(x_0,\frac{M\delta}{Nn^2})} \frac{1}{M^{2-d}}|x_0-y'|^{2-d}\rho(y,z_B)\rho(y,z_{A-B})M^{-d}dy'\nonumber\\
&=c_d\int_{B(x_0,\frac{M\delta}{Nn^2})} \frac{1}{M^{2}}|x_0-y'|^{2-d}\rho(y,z_B)\rho(y,z_{A-B})dy'.\label{firstbound}
\end{align}
Take $N=2M$. Note that $N> 1$, so this is consistent with the specification we gave earlier for $N$. Since $x_0\in C_n$, we have $B(x_0, \frac{M\delta}{Nn^2})=B(x_0, \frac{\delta}{2n^2})\subset C_{n-1}\subset C_0$ by \eqref{assumption2}. Therefore we may apply \eqref{assumption5a} and \eqref{firstbound} to get
\begin{equation}
\int_{B(x,\frac{\delta}{Nn^2})} g(x,y)\rho(y,z_B)\rho(y,z_{A-B})dy
\le \frac{K}{M^2}\int_{B(x_0,\frac{\delta}{2n^2})}g(x_0,y')\rho(y,z_B)\rho(y,z_{A-B})dy'.\label{first term}
\end{equation}
Choose $M\ge \sqrt{2K}\lor 1$ and 
combine the inequalities (\ref{first term}) and (\ref{secondterm}),  to see that
\begin{eqnarray}
\int_D g(x,y)\rho(y,z_B)\rho(y,z_{A-B})dy&\leq & \left(\tilde{K}+\tilde{B}n^{2(d-2)}\right)\int_{D} g(x_0,y)\rho(y,z_B)\rho(y,z_{A-B})dy\nonumber\\
\label{decompositionbound}
&&\mbox{ }+\frac{1}{2}\int_{ B(x_0,\frac{\delta}{2n^2})}g(x_0,y')\rho(y,z_B)\rho(y,z_{A-B})dy'.
\end{eqnarray}
Recall in the above inequality that $x$ and $x_0$ are arbitrarily selected points in $C_n$, $n\geq 2$, and the constants $\tilde{K}$, $\tilde B$ are independent of $x$, $x_0$ and $n$.

To complete the proof we will proceed with an induction argument. Let $\lambda$ be large enough that $\lambda \geq \phi$ (the constant in (\ref{assumption3})) and $\frac{\lambda^n}{2}\geq \tilde{K}+\tilde{B}n^{2(d-2)}$ for all $n\geq 1$.  Then we have that 
\[\rho(x,z)\leq \phi\rho(x_0,z)\leq \lambda \rho(x_0,z)\mbox{ for all  $x,x_0\in C_1,z\in\partial{D}$.}\]
 Let $n\geq 2$, and assume inductively that the following holds for all $k\leq n-1$, $x,x_0\in C_k$, $z_1,\ldots z_k\in \partial D$:
\[\rho(x,z_1,\ldots,z_k)\leq \lambda^k\rho(x_0,z_1,\ldots,z_k). \]
 Put $A=\{1,2,\ldots,n\}$.  Using the inequality (\ref{decompositionbound})  for all $x,x_0$ in $C_n$ we have that
 \begin{eqnarray}\label{induction}
\int_D g(x,y)\rho(y,z_B)\rho(y,z_{A-B})dy&\leq & \frac{\lambda^n}{2}\int_{D} g(x_0,y)\rho(y,z_B)\rho(y,z_{A-B})dy\nonumber\\
&&\mbox{ }+\frac{1}{2}\int_{ B(x_0,\frac{\delta}{2n^2})}g(x_0,y')\rho(y,z_B)\rho(y,z_{A-B})dy'.
\end{eqnarray} 
Note that if $y'\in B(x_0,\frac{\delta}{2n^2})$ then $y=x+\frac{1}{M}(y'-x_0)$ is in $B(x,\frac{\delta}{2Mn^2})$.  Since $M\geq 1$, $B(x,\frac{\delta}{2Mn^2})\subset C_{n-1}$ by (\ref{assumption2}).  Note also that $B(x_0,\frac{\delta}{2n^2})\subset C_{n-1}$.  By the inductive hypothesis, if $y\in B(x,\frac{\delta}{2Mn^2})$, and $y'\in B(x_0,\frac{\delta}{2n^2})$ then
\[\rho(y,z_B)\leq \lambda^{|B|}\rho(y',z_B)\mbox { and }\rho(y,z_B)\leq \lambda^{|A-B|}\rho(y',z_{A-B})\]
for all proper subsets $B$ of $A$.  Substituting this in (\ref{induction}) we get that
\begin{eqnarray*}
\int_D g(x,y)\rho(y,z_B)\rho(y,z_{A-B})dz&\leq & \frac{\lambda^n}{2}\int_{D} g(x_0,y)\rho(y,z_B)\rho(y,z_{A-B})dy\nonumber\\
&&\mbox{ }+\frac{1}{2}\int_{ B(x_0,\frac{\delta}{2n^2})}g(x_0,y')\lambda^{|B|}\lambda^{|A-B|}\rho(y',z_B)\rho(y',z_{A-B})dy'\\
&\leq&\lambda^n \int_D g(x_0,y)\rho(y,z_B)\rho(y,z_{A-B})dy.
\end{eqnarray*}
Hence
\begin{eqnarray*}
\rho (x,z_1,\ldots, z_n)&=&\sum_{B\subset A, B\neq \emptyset, A} \int_D g(x,y)\rho(y,z_B)\rho(y,z_{A-B})dy\\
&\leq&\lambda^n\sum_{B\subset A, B\neq \emptyset, A} \int_D g(x_0,y)\rho(y,z_B)\rho(y,z_{A-B})dy\\
&=&\lambda^n \rho(x_0,z_1,\ldots,z_n ).
\end{eqnarray*}
This establishes the result, when $d\ge 3$.

Essentially the same argument works when $d=2$. The 3G inequality in that case is classical; We will use the version stated in Theorem 1.1 of \cite{McC90}, namely that there exists $\theta$ for which 
\begin{equation}
\frac{g(x,y)}{g(x_0,y)}\leq \theta \frac{g(x,y)+g(x,x_0)+1}{g(x,x_0)}\mbox{ for all $x,x_0,y\in D$.}\label{assumption4in2d}
\end{equation}
Instead of \eqref{g_upperbound} and \eqref{g_lowerbound} we have constants $\tilde c$ and $c_2$ such that 
\begin{align*}
g(x,y)&\le c_2\log(1/|x-y|) +\tilde c\text{ for $x,y\in D$, and}\\
g(x,y)&\ge c_2\log(1/|x-y|)-c_2\log(1/\beta)\text{ if $x,y\in D$, $B(x,\beta)\subset D$.}
\end{align*}
As before, this gives \eqref{assumption5b} and 
$$
\frac{c_2\log(1/|x-x_0|)}{g(x,x_0)}\leq K
$$
for all $x,x_0\in C_0$, $x\neq x_0$, with $K=\max(Bc_2\log(2/\beta),1+Bc_2\log(1/\beta))$. 

Let $N>1$ and let $x,x_n\in C_n$ be distinct. By \eqref{assumption4in2d}, we have for $y\in D- B(x,\frac{\delta}{Nn^2})$ that
\begin{align*}
g(x,y)&\le
g(x_0,y)\theta\left(\frac{g(x,y)+1}{g(x,x_0)}+1\right)\\
&\le g(x_0,y)\theta\Big(B(1+\tilde c+c_2\log(1/|x-y|))+1\Big)\\
&\le g(x_0,y)\theta\Big(B(1+\tilde c+c_2\log(Nn^2/\delta))+1\Big).
\end{align*}
So for $\tilde K=\theta\Big(B(1+\tilde c+c_2\log(N/\delta))+1\Big)$ and $\tilde B=2\theta Bc_2$ we get
\begin{equation*}
\int_{D- B(x,\frac{\delta}{Nn^2})}g(x,y)\rho(y,z_B)\rho(y,z_{A-B})dy
\le (\tilde{K}+\tilde{B} \log n)\int_Dg(x_0,y)\rho(y,z_B)\rho(y,z_{A-B})dy \label{secondtermin2d}
\end{equation*}
in place of \eqref{secondterm}.

Take $N=2M$ and set $y'=x_0+M(y-x)$ as before. In place of \eqref{firstbound} we obtain that
\begin{align*}
\int_{B(x,\frac{\delta}{Nn^2})} g(x,y)\rho(y,z_B)\rho(y,z_{A-B})dy
&\leq \int_{B(x,\frac{\delta}{Nn^2})}\Big(c_2\log(1/|x-y|)+\tilde c\Big) \rho(y,z_B)\rho(y,z_{A-B})dy\nonumber\\
&=\int_{B(x_0,\frac{M\delta}{Nn^2})} \frac{1}{M^{2}} \Big(c_2\log(M/|x_0-y'|)+\tilde c\Big)\rho(y,z_B)\rho(y,z_{A-B})dy'\nonumber\\
&\le \int_{B(x_0,\frac{\delta}{2n^2})} \frac{K+B(\tilde c+c_2\log M)}{M^{2}} g(x_0,y')\rho(y,z_B)\rho(y,z_{A-B})dy'.\label{firstboundin2d}
\end{align*}
Now choose $M$ so large that $ \frac{K+B(\tilde c+c_2\log M)}{M^{2}}\le\frac12$, and $\lambda$ so that $\frac{\lambda^n}{2}\ge \tilde K +\tilde B\log n$ for $n\ge 1$, and the proof proceeds as before. 
\qed

\bibliography{SBM}
\bibliographystyle{imsart-nameyear}

\end{document}